\documentclass[12pt,a4paper]{amsart}

\usepackage[utf8]{inputenc}
\usepackage{fancyhdr}
\usepackage{indentfirst}
\usepackage{amsmath}
\usepackage{amssymb}
\usepackage{amsthm}
\usepackage[hidelinks,colorlinks,
breaklinks]{hyperref}
\usepackage{latexsym}
\usepackage{commath}
\usepackage{emptypage}
\usepackage{geometry}
\usepackage{enumerate}
\usepackage{enumitem}
\setlist[itemize]{label=\textbullet}
\usepackage{tikz-cd} 
\usepackage{comment}
\usepackage{mathtools}
\usepackage{leftidx}
\usepackage{todonotes}

\newtheorem{theorem}{Theorem}[section]
\newtheorem{lemma}[theorem]{Lemma}
\newtheorem{proposition}[theorem]{Proposition}
\newtheorem{cor}[theorem]{Corollary}

\theoremstyle{definition}

\newtheorem{definition}[theorem]{Definition}
\newtheorem{rmk}[theorem]{Remark}
\newtheorem{notation}[theorem]{Notation}
\newtheorem{example}[theorem]{Example}

\newcommand{\sslash}{\mathbin{/\mkern-5mu/}}

\newcommand{\eff}{\overline{\mathrm{Eff}}}

\newcommand{\oo}{\mathcal{O}}
\newcommand{\oa}{\mathcal{A}}
\newcommand{\ob}{\mathcal{B}}

\newcommand{\ou}{\mathcal{U}}
\newcommand{\oc}{\mathcal{C}}

\newcommand{\ol}{\mathcal{L}}

\newcommand{\ve}{\varepsilon}

\newcommand{\qq}{\mathbf{Q}}
\newcommand{\pp}{\mathbf{P}}

\DeclareMathOperator{\codim}{codim}

\DeclareMathOperator{\pic}{Pic}
\DeclareMathOperator{\nef}{Nef}
\DeclareMathOperator{\mov}{Mov}

\DeclareMathOperator{\proj}{Proj}

\author{Davide Ricci}

\usepackage[
backend=bibtex,
style=alphabetic,
sorting=nty,
maxbibnames=13,
doi=false, url=false, isbn=false
]{biblatex}

\begin{document}

\title{The Mori cone of certain Hassett spaces}

\begin{abstract}
In this paper, we prove that the Mori cone of Hassett spaces whose universal family is a $\pp^1$-bundle is generated by 1-dimensional strata. This extends the case of the symmetric GIT quotient $(\pp^1)^n\sslash PGL_2$ established by Bolognesi and Massarenti.
Along the way, we review how these spaces are naturally isomorphic to certain GIT quotients $(\pp^1)^n\sslash PGL_2$, characterize them as the targets of the birational contractions of the blow-up of $\pp^{n-3}$ at $n-1$ general points with $\qq$-factorial image, and deduce that their effective cone is likewise generated by strata.
\end{abstract}
\maketitle

\section{Introduction}

A central object in the Minimal Model Program is the study of the Mori cone $\overline{NE}(X)$ of a variety $X$, that is, the closure of the cone spanned by the classes of irreducible curves in the vector space of 1-cycles up to numerical equivalence. This cone governs several geometric properties of $X$. For example, extremal rays having negative intersection with the canonical divisor induce birational contractions to spaces with mild singularities.

A long-standing problem is determining the Mori cone of the moduli space $\overline M_{g,n}$ of stable genus $g$ curves with $n$ markings. This space compactifies $M_{g,n}$, the moduli space of smooth genus $g$ curves with $n$ markings, and is a stratified topological space: a $k$-codimensional stratum is an irreducible component parameterizing curves having at least $k$ nodes. The \textit{F-conjecture}, as stated in \cite{GKM}, predicts that $\overline{NE}(\overline M_{g,n})$ is generated by 1-dimensional strata. In \cite{GKM}, the authors reduce the conjecture to proving the genus 0 case. In particular, the conjecture for $\overline M_{g,n}$ is equivalent to the conjecture on $\overline M_{0,g+n}/S_g$, that is, the quotient of $\overline M_{0,g+n}$ by the action of the symmetric group $S_g$ that identifies the first $g$ marked points. The conjecture was proved for $\overline{M}_{0,n}$ for $n \le 7$ in \cite{keel1996contractibleextremalraysoverlinem0n}. For example, $\overline{M}_{0,5}$ is a del Pezzo surface of degree five and its Mori cone is well-known to be generated by (-1)-curves. Recently, Fedorchuk and Mellit proved in \cite{FM} that the conjecture on $\overline M_{0,n}$ for all $n$ holds if and only if the conjecture on $\overline M_{0,n}/S_n$ holds for all $n$. They also settled the case $\overline M_{0,8}$. On the other hand, in \cite{CTMDS},\cite{Gonzalezkaru},\cite{Hausen2016OnBU},\cite{mullane2025isoresidualfibrationsbirationalgeometry}, it was proved that $\overline M_{0,n}$ is not a Mori dream space for $n\ge8$. Mori dream spaces are a particularly well-behaved class of spaces for the minimal model program and include Fano varieties and, more generally, varieties of Fano type. Moreover, in \cite{keel1996contractibleextremalraysoverlinem0n} it was proved that $-K_{\overline M_{0,n}}$ is not effective for $n\ge8$. These facts show the complicated nature of this space, making it harder to study. In this paper, we approach a similar problem for simpler spaces.
\begin{definition}\label{Hasspace}
    Given a weight tuple $\oa=(a_1,\dots,a_n)$, $a_i\in\qq$, $0<a_i\le1$, $\sum a_i>2$, a \textit{weighted pointed $\oa$-stable rational curve} is an arithmetic genus zero connected nodal curve with $n$ smooth points $(C,p_1,\dots,p_n)$ such that $\omega_C(\sum a_ip_i)$ is ample and, every time some of the points coincide, the sum of their weights is at most one. We call \textit{Hassett space} the fine moduli space  $\overline M_\oa$ that parametrizes $\oa$-stable curves. We say that $\overline M_\oa$ is \textit{small} if its universal family is a $\pp^1$-bundle.
\end{definition}
These spaces were introduced by Hassett in \cite{hassett2002modulispacesweightedpointed} and provide different compactifications for $M_{0,n}$. The spaces $\overline M_\oa$ are generally easier to study, as they are targets of blow-up morphisms $\overline M_{0,n}\to\overline{M}_\oa$, and carry an analogous stratification of the boundary. Hence, it is natural to ask whether it is true that the Mori cone of $\overline{M}_\oa$ is generated by 1-dimensional strata. Moreover, the F-conjecture implies the same statement for any of these spaces $\overline M_\oa$. In this paper, we focus on small Hassett spaces, which are the simplest case. The main results are the following.
\begin{theorem}\label{maincor}\begin{enumerate}
    \item\label{one} The Mori cone of a small Hassett space $\overline{NE}(\overline M_{\oa})$ is finitely generated by 1-dimensional strata.
    \item\label{two} The cone of effective divisors of a small Hassett space $\overline{\textup{Eff}}(\overline M_{\oa})$ is finitely generated by 1-codimensional strata. 
    \item Small Hassett spaces are Mori dream spaces. 
    \item  Over the complex numbers, small Hassett spaces are of Fano type.
    \item Small Hassett spaces can be identified with the GIT quotients $(\pp^1)^n\sslash PGL_2$ with no strictly semi-stable points.
    \item\label{th126} The targets of the birational contractions with $\qq$-factorial image of the blow-up of $\pp^{n-3}$ at $n-1$ points in general position $Bl_{n-1}\pp^{n-3}$ are precisely the small Hassett spaces. In particular, they are all smooth.
\end{enumerate}
\end{theorem}
This reproves and extends a result presented by Bolognesi and Massarenti in \cite{Bolognesi_Massarenti}: the Mori cone of the GIT quotients $(\pp^1)^n\sslash PGL_2$ with the symmetric linearizations is generated by 1-dimensional strata. In particular, when $n$ is odd, the space $(\pp^1)^n\sslash PGL_2$ is isomorphic to $\overline{M}_\oa$ where $\oa=(2/n+\ve,\dots,2/n+\ve)$, $0<\ve\ll1$. When $n$ is even, $(\pp^1)^n\sslash PGL_2$ is singular, its Picard rank is one, and $\overline{M}_\oa$ is its Kirwan desingularization, but it is no longer small. In all cases, $\overline{M}_\oa$ is an instance of the $\overline{M}_{p,q}$ spaces that appeared in \cite{CT3}, i.e.  the Hassett spaces with $a_1=\dots=a_p=a+\ve$, $a_{p+1}=\dots=a_n=b$, $p+q=n$, $pa+qb=2$, $0<b,\ve\ll1$. These $\overline{M}_{p,q}$ spaces are small as long as at least one among $p$ and $q$ is odd. When the $\overline{M}_{p,q}$ spaces are not small, they fit in the next easiest class of Hassett spaces to study, that is, when the stable curves have at most two irreducible components. We leave open the question of the generators for the Mori cones for these spaces.

In \cite{Bolognesi_Massarenti}, the Mori chamber decomposition of $Bl_{n-1}\pp^{n-3}$ and its link with the GIT chamber decomposition for the action of $PGL_2$ on $(\pp^1)^n$ were also studied. In Section \ref{sectiongit}, we revise, with a general setting, how small Hassett spaces arise as quotients $(\pp^1)^n\sslash PGL_2$ for certain linearizations. In Section \ref{sectionbirational}, we reprove the Mori chamber decomposition of $Bl_{n-1}\pp^{n-3}$ and use it to prove Theorem \ref{maincor}.(\ref{th126}).

\subsection*{Conventions} We work over an arbitrary algebraically closed field unless the complex field is specified.

\subsubsection*{Acknowledgments}
The author is grateful to his supervisors, Ana-Maria Castravet and Enrica Floris, for their support and guidance throughout the development of this project. The author thanks Pietro Capozzolo for attempting a joint computational approach to more challenging cases, even though the numerical results were inconclusive.
The author has been co-funded by the European Union’s Horizon Europe research and innovation program under the Marie Skłodowska-Curie Grant Agreement No 101126554.

\section{Background on Hassett spaces}\label{background}

In this section, we recall some background notions on Hassett spaces from \cite{hassett2002modulispacesweightedpointed}, and we introduce some straightforward properties of small Hassett spaces (cf. Definition \ref{Hasspace}) in Lemma \ref{lemmasmall} with some examples.
\begin{rmk}
    Fix two weight tuples $\oa=(a_1,\dots,a_n)$ and $\ob=(b_1,\dots,b_n)$ as in Definition \ref{Hasspace}. If $a_i\ge b_i$ for each $i$, by \cite[Theorem 4.1]{hassett2002modulispacesweightedpointed} there is a \textit{reduction morphism} $\rho=\rho_{\ob,\oa}:\overline M_{\oa}\to{\overline M}_{\ob}$, which is birational and isomorphism on $ M_{0,n}$.
\end{rmk}
\begin{definition}\label{boundariesdef}
Given a subset $I\subset\{1,\dots,n\}$ such that $2\le|I|\le n-2$, the image $\delta_{I,I^c}$ of the gluing morphism $\overline M_{0,I\cup\{*\}}\times\overline M_{0,I^c\cup\{\star\}}\to \overline M_{0,n}$ is an irreducible component of $\overline{M}_{0,n}\setminus M_{0,n}$, which we denote by $\delta_I$ for simplicity. These components are called \textit{$1$-codimensional strata}. Similarly, \textit{$k$-codimensional strata} are defined as the images of $k$ consecutive gluings.

Strata in $\overline{M}_{\oa}$ are defined as images of those in $\overline M_{0,n}$ via the reduction morphism. Precisely, \textit{$k$-codimensional strata} in $\overline M_\oa$ are the images via $\overline M_{0,n}\to \overline M_{\oa}$ of strata (of any dimension) in $\overline{M}_{0,n}$ that have codimension $k$ in $\overline{M}_\oa$. A 1-codimensional stratum in $\overline{M}_\oa$ is determined by a subset $I\subset\{1,\dots,n\}$ such that $2\le|I|\le n-2$ is again denoted $\delta_I$.
\end{definition}
\begin{rmk}\label{divisors}
The dimension of the strata might drop after a reduction morphism $\rho:\overline M_{\oa}\to\overline{M}_\ob$. It is proved in \cite[Proposition 4.5]{hassett2002modulispacesweightedpointed} that every time $\sum_{i\in I}a_i>1$ and $\sum_{i\in I^c}a_i>1$, the divisor $\delta_I$ in $\overline M_{\oa}$ is contracted by the reduction morphism $\rho:\overline M_{\oa}\to\overline{M}_\ob$ if and only if $|I|>2$ and $\sum_{i\in I}b_i\le1$, or if the same happens to $I^c$. 

We list the possible ways in which $1$-codimensional strata can appear. Given $I$, the possibilities for $\delta_I$ are:
\begin{enumerate}
    \item\label{type1} If $\sum_{i\in I}a_i>1$ and $\sum_{i\in I^c}a_i>1$, the stratum $\delta_I$ is a divisor and is the image of the gluing
\[
\overline M_{\oa_I}\times \overline M_{\oa_{I^c}}\to \overline M_{\oa},
\]
with $\oa_I=(a_{i_1},\dots,a_{i_r},1)$, $I=\{i_1,\dots,i_r\}$,  $\oa_{I^c}=(a_{j_1},\dots,a_{j_{n-r}},1)$, $I^c=\{j_1,\dots,j_{n-r}\}$. In this case, its points correspond to curves with two irreducible components that intersect transversely at a point, one with points marked by $I$ and the other by $I^c$.
\item\label{type2} If $\sum_{i\in I}a_i\le1$ and $|I|=2$, the stratum $\delta_I$ is again a divisor. If $I=\{i_1,i_2\}$, its points correspond to irreducible curves where the marked points $p_{i_1}$ and $p_{i_2}$ are identified. This time, $\delta_I$ is isomorphic to the Hassett space $\overline{M}_{\oa'}$ where $\oa'=(a_1,\dots,\hat a_{i_1},\dots,\hat a_{i_2},\dots,a_n,a_{i_1}+a_{i_2})$, with $a_{i_1}$ and $a_{i_2}$ omitted.
\item In the case where $\sum_{i\in I}a_i\le1$ and $|I|>2$, we have $\delta_I$ is a $(|I|-1)$-codimensional stratum and not a divisor. 
\end{enumerate}
\end{rmk}
\begin{definition}\label{divisortype}
    Let $\overline{M}_\oa$ be a Hassett space. We say that a divisor $\delta_I$ is \begin{enumerate}
        \item of \textit{type-1} if it as in point \ref{type1} of Remark \ref{divisors} above, and
        \item of \textit{type-2} if it is as in point \ref{type2} of Remark \ref{divisors} above.
    \end{enumerate}
\end{definition}
\begin{notation}
Let $S\subset\qq^n$ be a connected topological subspace. A \textit{chamber decomposition} of $S$ consists of a finite set $\{H_1,\dots,H_r\}$ of hyperplanes of $\qq^n$, called \textit{walls}. The connected components of $S\setminus\bigcup H_i$ are called the \textit{chambers} of the decomposition.
\end{notation}
\begin{definition}[{{{\cite[Chapter 5]{hassett2002modulispacesweightedpointed}}}}]\label{domainD}
Denote the set whose points correspond to weight tuples that are admissible for Definition \ref{Hasspace} by \[
\mathcal{D}_{0,n}=\{(a_1,\dots,a_n)\in\qq^n\:|\:0<a_i\le1\text{ and }\sum a_i>2 \}.
\]
As a topological space, $\mathcal{D}_{0,n}$ has two natural chamber decompositions:
\begin{enumerate}
    \item the \textit{fine} chamber decomposition, whose walls are the hyperplanes $\sum_{i\in I}a_i=1$ for $I\subset\{1,\dots,n\}$ and $2\le|I|\le n-2$, and
    \item the \textit{coarse} chamber decomposition, whose walls are the hyperplanes $\sum_{i\in I}a_i=1$ for $I\subset\{1,\dots,n\}$ and $2 <|I|\le n-2$.
\end{enumerate}
\end{definition}
As proved in \cite[Proposition 5.1]{hassett2002modulispacesweightedpointed}, the variation of the weight tuple $\oa$ in the coarse chamber decomposition describes the variation of the moduli space $\overline{M}_\oa$. In contrast, the fine chamber decomposition describes the variation of the moduli functor, i.e., of the universal family $\ou\to\overline{M}_\oa$. A reduction morphism $\rho:\overline M_{\oa}\to\overline{M}_\ob$ such that $\overline M_{\oa}$ and $\overline M_{\ob}$ are in the same coarse chamber is an isomorphism. The only thing that can happen is that some type-1 divisors become type-2 divisors.

\begin{example}[Kapranov's construction]\label{exom}
This example reports some realizations of Hassett spaces as standard varieties. Among these, we find $Bl_{n-1}\pp^{n-3}$, that is, the blow-up of $\pp^{n-3}$ at $n-1$ general points, which will be central in Section \ref{sectionbirational} and Section \ref{Lastsec}. A more detailed discussion can be found in \cite[Construction 6.1]{hassett2002modulispacesweightedpointed}.

Fix $n\ge4$ and define the $n$-tuples $\oa_k=(\frac{1}{n-2-k},\dots,\frac{1}{n-2-k},1)$ for $k=0,\dots,n-4$. The Hassett space $\overline M_{\oa_0}$ is isomorphic to $\pp^{n-3}$, and, for $k\ge1$, $\overline M_{\oa_k}$ is isomorphic to the iterated blow-up of $\pp^{n-3}$ at $n-1$ points in general position, at all the proper transforms of the lines spanned by pairs of points, and so until we blow-up all the proper transforms of the $(k-1)$-dimensional linear subspaces spanned by $k$ of the points. Moreover, the reduction morphism $\rho_{\oa_{k-1},\oa_k}:\overline M_{\oa_{k}}\to\overline M_{\oa_{k-1}}$ is isomorphic to the blow-up morphism of the proper transforms of the $(k-1)$-dimensional linear subspaces. 

We sketch how the isomorphism between $Bl_{n-1}\pp^{n-3}\to\pp^{n-3}$ and $\overline{M}_{\oa_1}\to\overline M_{\oa_0}$ can be set theoretically written down. 
Any points of $\overline M_{\oa_0}$ can be represented as $(\pp^1,x_1,\dots,x_{n-2},0,\infty)$. This gives a bijection between $\overline M_{\oa_0}$ and $\pp^{n-3}$. Now, set the $n-1$ general points of $\pp^{n-3}$ as $P_1=[1:0:\dots:0] ,\dots,P_{n-2}=[0:\dots:0:1]$, and $P_{n-1}=[1:\dots:1]$. Under the bijection above, they are mapped to:
    \begin{align}
        &P_i\textup{ for }i=1,\dots,n-2&\longleftrightarrow &\quad0=x_1=\dots=\hat x_i=\dots=x_{n-2},\\
        &P_{n-1}&\longleftrightarrow&\quad x_1=\dots=x_{n-2}.
    \end{align}
We consider the blow-up $Bl_{n-1}\pp^{n-3}$ of $\pp^{n-3}$ at these points, and we denote by $E_i$ the exceptional divisor at $P_i$ and $\widetilde\Lambda_{ij}$ the strict transform of the hyperplane passing through all points except $P_i$ and $P_j$. It is immediate to see that the boundary divisor $\delta_{\{i,n\}}$ is of type and 1 corresponds to $E_i$, and $\delta_{\{i,j\}}$ is of type 2 and corresponds to $\widetilde\Lambda_{ij}$.     
\end{example}

Recall from Definition \ref{Hasspace} that we defined small Hassett spaces as those whose universal family is a $\pp^1$-bundle. The following lemma contains conditions that are useful for understanding these spaces from the weight tuple.
\begin{lemma}\label{lemmasmall}
    Let $\overline{M}_\oa$ be a Hassett space.
    \begin{enumerate}
        \item\label{lm4} The following are equivalent:\begin{enumerate}
            \item $\overline{M}_\oa$ is small, 
            \item $\overline{M}_\oa$ has only type-2 divisors, and
            \item there is no subset $I\subset\{1,\dots,n\}$ with $\sum_{i\in I}a_i>1$ and $\sum_{i\in I^c}a_i>1$.
        \end{enumerate} 
        \item\label{lm1} There exists a reduction morphism $\overline{M}_\oa\to\overline{M}_\ob$ with $\overline{M}_\ob$ that is small.
        \item\label{lm2} $\overline{M}_\oa$ is isomorphic to a small Hassett space if and only if there is no subset $I\subset\{1,\dots,n\}$ such that $2<|I|<n-2$, $\sum_{i\in I}a_i>1$ and $\sum_{i\in I^c}a_i>1$.
        \item\label{lm3} If $\overline{M}_\oa$ is small, then every stratum of $\overline{M}_\oa$ is isomorphic to a small Hassett space.
    \end{enumerate}
\end{lemma}
\begin{proof}
For (\ref{lm4}), the universal family of a Hassett space is a $\pp^1$-bundle if and only if it has only type-2 divisors by definition. The second equivalence is exactly the condition for not having type-1 divisors, see Remark \ref{divisors}.

    Up to perturbing $\oa$ within the fine chamber of $\oa$, we can assume that there is no subset $I\subset\{1,\dots,n\}$ such that $\sum_{i\in I}a_i=\frac{1}{2}\sum_{i=1}^na_i$. Define $t=\frac{2+\ve}{\sum a_i}$ for $0<\ve\ll1$ and $\ob=(b_i=t\cdot a_i)$. Notice that $a_i\ge b_i$ for all $i$. We have $\sum b_i=t\sum a_i=2+\ve$. At this point, either $\overline{M}_\ob$ is small, and in this case we are done, or there exists $I\subset\{1,\dots,n\}$, such that $\sum_{i\in I}b_i=1+\eta_1^I$ and $\sum_{i\in I^c}b_i=1+\eta_2^I$, for $\eta_1^I+\eta_2^I=\ve$ and $0<\eta_1^I\ne \eta_2^I<\ve$. In the second case, define $\eta=\max_I\{\eta_i^I\}$, and define $\ob'=(b_1-\eta,b_2,\dots,b_n)$. Now, $\overline{M}_{\ob'}$ is small and this proves (\ref{lm1}).

    For (\ref{lm2}), if $\overline{M}_\oa$ is isomorphic to a small Hassett space $\overline{M}_\ob$, then they live in the same coarse chamber and the weight condition is necessary. On the other hand, given $\oa$ as in the hypothesis, a way to find an isomorphism $\overline{M}_\oa\to\overline{M}_\ob$ with $\overline{M}_\ob$ small is the following. If $I=\{i_1,i_2\}$ we have $a_{i_1}+a_{i_2}>1$ and $\sum_{i\in I^c}a_i>1$, define $\oa'=(a_i')$ with \[a_{i_1}'=a_{i_1}-\frac{a_{i_1}+a_{i_2}-1}{2},\qquad a_{i_2}'=a_{i_2}-\frac{a_{i_1}+a_{i_2}-1}{2},\qquad a_i'=a_i\textup{ for }i\in I^c.\]
In this way, we have $a_{i_1}'+a_{i_2}'=1$ and $\sum_{i\in I^c}a'_i>1$. Precisely, as explained in \ref{divisors}, the reduction $\overline M_\oa\to\overline M_{\oa'}$ is an isomorphism, maps the type-1 divisor $\delta_I$ in $\overline M_\oa$ to the corresponding type-2 divisor $\delta_I$ in $\overline M_{\oa'}$, and the number of type-1 divisors in $\overline M_{\oa'}$ has decreased by 1 from $\overline M_{\oa}$. We can iterate this operation starting from $\oa'$ and reach in a finite number of steps a small tuple $\ob$.

    We already mentioned in Remark \ref{divisors} that a type-2 divisor $\delta_{\{i_1,i_2\}}$ is isomorphic to the Hassett spaces $\overline{M}_{\oa'}$ with $\oa'=(a_1,\dots,\hat a_{i_1},\dots,\hat a_{i_2},\dots,a_n,a_{i_1}+a_{i_2})$. It is immediate from (\ref{lm4}) to see that $\overline{M}_{\oa'}$ is small too. Since 2-dimensional strata in $\overline{M}_{\oa}$ can be seen as 1-dimensional strata inside boundary divisors, we prove (\ref{lm3}) by induction on the codimension.
\end{proof}

\begin{example}\label{qusexample}
Using Lemma \ref{lemmasmall}.(\ref{lm2}),(\ref{lm3}), it is easy to determine whether $\overline{M}_\oa$ is small given the tuple $\oa$. For example, when $n\le 5$, all Hassett spaces are isomorphic to small Hassett spaces. For any $n$, the Hassett space $\overline{M}_{\oa_1}\cong Bl_{n-1}\pp^{n-3}$ of Example \ref{exom} is isomorphic to a small Hassett space but is not small itself. Define the tuple $\oa=(a+\ve,\dots,a+\ve,b,\dots,b)$ where $a+\ve$ appears $p\ge3$ times, $b$ appears $q\ge0$ times, at least one between $p$ and $q$ is odd, $ap+bq=2$, and $0<b,\ve\ll1$. Then $\overline M_{p,q}:=\overline{M}_\oa$ is small. This space played a central role in \cite{CT3}. When $p=n$ is odd and $q=0$, Proposition \ref{proptypical} of the next section will prove that $\overline{M}_{p,0}$ is isomorphic to the GIT quotient $(\pp^1)^n\sslash_{\oo(1,\dots,1)}PGL_2$.
\end{example}

\section{Small Hassett spaces and GIT quotients of $(\pp^1)^n$}\label{sectiongit}

Small Hassett spaces can also be interpreted as GIT quotients of $(\pp^1)^n$ by $PGL_2$ acting diagonally, cf. Proposition \ref{proptypical}. All the details about GIT quotients and this particular action can be found in \cite[Chapter 11]{Dolgachev_2003}. These GIT quotients depend on the choice of a linearized ample line bundle on $(\pp^1)^n$, called linearization. The ample line bundles of $(\pp^1)^n$ are of the form $\oo(b_1,\dots,b_n)$, for some positive integers $b_i$, and linearize for the action of $PGL_2$ if and only if $b_1+\dots+b_n$ is even. However, the resulting GIT quotient does not depend on the choice of the linearization up to multiples. Therefore, the data of a linearization for the action of $PGL_2$ is encoded in a tuple of rational numbers $\ob=(b_1,\dots,b_n)$ with $0<b_i\le1$ and $\sum b_i=2$. The semi-stable points in $(\pp^1)^n$ are tuples $(p_1,\dots,p_n)$ such that every time $p_{i_1}=\dots=p_{i_r}$ we have $\sum_{j=1}^rb_{i_j}\le1$, while stable points require the inequality to be strict. 
\begin{notation}\label{vgitdomain}
Let $\ob$ be a linearization for the action of $PGL_2$ on $(\pp^1)^n$. We denote by $\Sigma_\ob=(\pp^1)^n\sslash_\ob PGL_2$ the GIT quotient with respect to the linearization provided by $\ob$. We denote by
\[
\mathcal{V}_n=\{(b_1,\dots,b_n)\in\qq^n\:|\:0<b_i\le1\text{ and }\sum b_i=2 \}.\]
the topological space whose points correspond to the possible linearizations $\ob$. We will call \textit{VGIT chamber decomposition} of $\mathcal{V}_n$ the chamber decomposition that has the hyperplanes $\sum_{i\in I}b_i=1$ for $I\subset\{1,\dots,n\}$ as walls.

Let $\mathrm{Amp}((\pp^1)^n))$ be the cone spanned by ample divisors in $\pic((\pp^1)^n)_\qq$.
We define a \textit{normalization} map $N:\mathrm{Amp}((\pp^1)^n)\to\mathcal{V}_n$ that associates the ample line bundle given by the tuple $\oa=(a_i)$ with the linearization $N(\oa)=(\frac{2a_i}{\sum a_j})$.
\end{notation} 
\begin{definition}[{{\cite[Section 8]{hassett2002modulispacesweightedpointed}}}]
    Let $\ob$ be a linearization for the action of $PGL_2$ on $(\pp^1)^n$. If stability and semi-stability coincide, i.e., if there is no subset $I\subset\{1,\dots,n\}$ such that $\sum_{i\in I}b_i=1$, the linearization $\ob$ is called \textit{typical}. 
\end{definition}
The chambers of the VGIT decomposition of $\mathcal{V}_n$ represent the regions of $\mathcal{V}_n$ where the stability does not change: if $\ob$ and $\ob'$ are typical, the quotients $\Sigma_\ob$ and $\Sigma_{\ob'}$ are isomorphic if and only if $\ob$ and $\ob'$ lie in the same chamber. This is treated in details in \cite{Dolgachev_2003}. In this setting, typical linearizations are those that lie in a VGIT chamber. Inside $\qq^n$, the space $\mathcal{V}_n$ consists of the face $\sum a_i=2$ of the closure of the domain $\mathcal{D}_{0,n}$ from Definition \ref{domainD}. The hyperplanes of the fine and the coarse chamber decompositions of $\mathcal{D}_{0,n}$ induce the same chamber decomposition on $\mathcal{V}_n$.

The next goal is to stress the link between Hassett spaces with small weights and the GIT quotients of $(\pp^1)^n$ arising from typical linearizations. Specifically, we combine \cite[Theorem 8.2]{hassett2002modulispacesweightedpointed} and \cite[Lemma 2.7]{CT3}.

\begin{definition}
    Given $\oa\in\mathcal{D}_{0,n}$ and $\ob\in\mathcal{V}_n$, we say that \textit{$\oa$-stability and $\ob$-stability are the same} to mean that $\sum_{i\in I} a_i\le 1$ if and only if $\sum_{i\in I} b_i\le 1$ for any $I\subset\{1,\dots,n\}$.
\end{definition}

\begin{proposition}\label{proptypical}
Let $\ob\in\mathcal{V}_n$ be a typical linearization, and let $\oa\in\mathcal{D}_{0,n}$.
\begin{enumerate}
    \item\label{proptyp1} 
    If $\oa$-stability and $\ob$-stability are the same, there is a natural isomorphism $\Sigma_\ob\xrightarrow{\sim}\overline{M}_\oa$.
    \item\label{proptyp2} There exists a small Hassett space $\overline{M}_{\ob_\ve}$ isomorphic to $\Sigma_\ob$.
    \item\label{proptyp3} If $\overline{M}_\oa$ and $\Sigma_\ob$ are isomorphic, then $\ob$ lies in the closure of the coarse chamber of $\oa$ in $\mathcal{D}_{0,n}$.
    \item\label{proptyp4} There is a bijection between the coarse chambers of $\mathcal{D}_{0,n}$ that contain small Hassett spaces and the chambers of the VGIT decomposition for the action of $PGL_2$ on $(\pp^1)^n$. Corresponding Hassett spaces and GIT quotients are isomorphic.
\end{enumerate}
\end{proposition}
\begin{proof}
Let $Y=(\pp^1)^n$. The trivial $\pp^1$-bundle $Y\times\pp^1\to Y$ has $n$ sections $s_1,\dots,s_n$ determined by $s_i(p_1,\dots,p_n)=((p_1,\dots,p_n),p_i)$. This bundle restricts to the semi-stable locus $Y^{ss}$ for the $PGL_2$-action, and the fibers here are $\oa$-stable curves. If $\oa$-stability and $\ob$-stability are the same, by the universal property of $\overline M_\oa$, there is a morphism $Y^{ss}\to\overline M_\oa$. Notice that $\overline{M}_\oa$ has a trivial $PGL_2$-action as its points parametrize isomorphism classes of $\oa$-stable curves. The morphism produced above is therefore a $PGL_2$-equivariant morphism and factors as $Y^{ss}\to\Sigma_\ob\to\overline M_\oa$. Since both $\Sigma_\ob$ and $\overline M_\oa$ parametrize curves with the same stability condition, the morphism between them is bijective. Moreover, it is birational as it is an isomorphism when restricted to $M_{0,n}$. Therefore, since the target is smooth, it is an isomorphism and (\ref{proptyp1}) is proved.

Given $\ob=(b_1,\dots,b_n)$, for a positive and sufficiently small $\ve$, since $\ob$ is typical, the tuple \[
    \ob_\ve=(b_1+\ve,\dots,b_n+\ve)
\]
is small and $\ob_\ve$-stability and $\ob$-stability are the same. Hence, by (\ref{proptyp1}) there is an isomorphism $\overline{M}_{\ob_\ve}\xrightarrow{\sim}\Sigma_\ob$. This proves (\ref{proptyp2}). Moreover, $\ob$ lies in the closure of the coarse chamber of $\ob_\ve$ and any Hassett space isomorphic to $\Sigma_\ob$ must be constructed with a tuple in the same coarse chamber of ${\ob_\ve}$ in $\mathcal{D}_{0,n}$. This proves (\ref{proptyp3}).

The only thing remaining to prove (\ref{proptyp4}) is to associate a typical linearization with a small Hassett space. Let $\overline{M}_{\oa}$ be small. Define $\ob_\oa=N(\oa)$, i.e., as the linearization obtained by setting $b_j=\frac{2a_j}{\sum a_i}$ (see Notation \ref{vgitdomain}). Notice that $\sum_{i\in I} b_i\ge1$ implies $\sum_{i\in I} a_i>1$. This implies that $\ob_\oa$ is typical and $\oa$-stability and $\ob_\oa$-stability are the same, and $\overline{M}_\oa$ and $\Sigma_{\ob_\oa}$ are isomorphic.
\end{proof}

\section{Small Hassett spaces as birational contractions of {{{$Bl_{n-1}\pp^{n-3}$}}}}\label{sectionbirational}
\subsection{Main results of the section}

The purpose of this section is to characterize small Hassett spaces as birational contractions of $Bl_{n-1}\pp^{n-3}$ associated with Mori chambers. From this, we deduce that small Hassett spaces are Mori dream spaces and, over the complex numbers, of Fano type. This highlights the simplicity of these spaces in contrast with $\overline{M}_{0,n}$. As a further consequence, our characterization shows that the birational contractions of $Bl_{n-1}\pp^{n-3}$ with $\qq$-factorial image are smooth; this result was proved for the small $\qq$-factorial modifications in \cite[Theorem 2.32]{Bolognesi_Massarenti}. The main results of the section are summarized in Proposition \ref{proposition}. The VGIT theory of $(\pp^1)^n$ with the action of $PGL_2$ and the Mori chamber decomposition of $Bl_{n-1}\pp^{n-3}$ play a central role in the proof.

The references for the following definitions are \cite{HK} and \cite{okawa}. 
\begin{definition}
    Let $f:X\dashrightarrow Y$ be a birational map between normal projective varieties, and let $X\xleftarrow{q}Z\xrightarrow{p}Y$ be a resolution of indeterminacies where $Z$ is projective and smooth. We say that $f$ is
    \begin{enumerate}
        \item a \textit{birational contraction}, if every $q$-exceptional divisor is also $p$-exceptional, and
        \item a \textit{small $\qq$-factorial modification}, or \textit{SQM} for short, if $Y$ is $\qq$-factorial and $p$ and $q$ contract the same divisors.
    \end{enumerate}
\end{definition}
    Let $f:X\dashrightarrow Y$ be a birational map between $\qq$-factorial projective varieties. Requiring $f$ to be a birational contraction is equivalent to $f$ being surjective in codimension 1. Requiring it to be an SQM is equivalent to it being an isomorphism in codimension 1.

\begin{definition}
    A $\qq$-factorial projective variety $X$ is called \textit{Mori dream space} if 
    \begin{enumerate}
        \item\label{mds1} $H^1(X,\oo_X)=0$,
        \item\label{mds2} $\nef(X)$ is the convex hull of finitely many semi-ample line bundles, and
        \item There is a finite collection of SQMs $f_i: X\dashrightarrow X_i$ such that each $X_i$ satisfies \ref{mds1} and \ref{mds2}, and $\mov(X)$ is the union of the $f_i^*(\nef(X_i))$.
    \end{enumerate}
\end{definition}
Recall that $\mov(X)$ is the cone spanned by the classes of movable divisors in $N^1(X)_\qq$, i.e., with a stable base locus of codimension at least 2, and that $\eff(X)$ is the closure of the cone spanned by the classes of effective divisors in $N^1(X)_\qq$.
\begin{proposition}[\cite{HK}, Proposition 1.11]\label{MDSprop}
    Let $X$ be a Mori dream space. Then,
    \begin{enumerate}
        \item the effective cone $\eff(X)$ is rational polyhedral,
        \item there is a decomposition of $\eff(X)$ into finitely many closed convex chambers, whose interiors are disjoint and called the Mori chambers of $X$,
        \item the Mori chambers of $\eff(X)$ are in 1-1 correspondence with the birational contractions of $X$ with $\qq$-factorial image,
        \item the birational contraction of $X$ corresponding to a Mori chamber $C$ is the map $\psi_{mD}:X\dashrightarrow\proj R(X,mD)$ for every $D$ in $C$ and $m\gg0$, and
        \item the targets of the birational contractions of $X$ are Mori dream spaces.
    \end{enumerate} 
\end{proposition}
\begin{example}\label{extremeblowup}
    It is well-known that $Bl_{n-1}\pp^{n-3}$ is a Mori dream space, and its Cox ring is the homogeneous coordinate ring of the grassmannian $\mathrm{Gr}(2,n)$, see for example  \cite{Castravet_Tevelev_2006}.
    Also, its effective cone can be easily described. Let $E_i$ be the exceptional divisor at the $i$-th point, and let $\widetilde{\Lambda}_{ij}$ be the strict transform of the hyperplane spanned from all the points except for $p_i$ and $p_j$, $i,j\ne n$. Numerically, $\widetilde{\Lambda}_{ij}\equiv H-E_i-E_j$ where $H$ is the pullback of a hyperplane class. The divisors $\widetilde{\Lambda}_{ij}$ for all $i,j$ and $E_i$ for all $i$ generate the extremal rays of $\overline{\textup{Eff}}(Bl_{n-1}\pp^{n-3})$.

    Set $\oc=(\frac{1}{n-3},\dots,\frac{1}{n-3},1)$ and recall from Example \ref{exom} that $\overline{M}_\oc\cong Bl_{n-1}\pp^{n-3}$, and that the isomorphism makes the boundary divisors $\delta_{\{i,n\}}$ correspond to the exceptional divisors $E_i$ and $\delta_{\{i,j\}}$ to $\widetilde{\Lambda}_{ij}$. We deduce that $\overline{\textup{Eff}}(\overline{M}_{\oc})$ is finitely generated by 1-codimensional strata, which are precisely the extremal rays of the Mori cone. 
\end{example}

The following proposition is the main result of this section. Its proof requires some preliminary results, which we develop in the next subsection.

Let $X$ be a $\qq$-factorial projective variety. Recall that $X$ is of \textit{Fano type} if there exists an effective divisor $D$ such that the pair $(X,D)$ is klt and $-(K_X+D)$ is ample (see \cite[Definition 2.8]{Kollar_2013} for the definition of klt pair).

\begin{proposition}\label{proposition}Denote $X=Bl_{n-1}\pp^{n-3}$.
\begin{enumerate}
    \item\label{pr1} There is a bijection between the Mori chambers of $X$ and the coarse chambers of $\mathcal{D}_{0,n}$ associated with small Hassett spaces.
    \item\label{pr2} Small Hassett spaces and targets of corresponding birational contractions are isomorphic.
    \item\label{pr3} Small Hassett spaces are Mori dream spaces. 
    \item\label{pr4} Over the complex numbers, small Hassett spaces are of Fano type.
    \item\label{pr5} The targets of the birational contractions of $X$ with $\qq$-factorial image are smooth.
    \item\label{pr6} If $\overline M_{\oa}$ is a small Hassett space, the birational contraction $X\dashrightarrow \overline M_{\oa}$ is an SQM if and only if there are no subsets $I\subset\{1,\dots,n\}$ with $|I|=n-2$ and $\sum_{i\in I}a_i\le1$.
\end{enumerate}
\end{proposition}
\begin{rmk}
    Recall from Example \ref{qusexample} the isomorphism between the symmetric GIT quotient $(\pp^1)^n\sslash_{\oo(1,\dots,1)} PGL_2$ and the small Hassett space $\overline{M}_{p,q}$ for $p=n$ odd and $q=0$. In \cite[Lemma 6.6]{torres}, it was proved that, in characteristic 0, the linearization $(2/n,\dots,2/n)$ for $n$ odd is the only typical linearization with unstable locus of codimension at least 2 such that the quotient $(\pp^1)^n\sslash_\ob PGL_2$ is Fano. On the other hand, as a consequence of Proposition \ref{proposition} and Proposition \ref{proptypical}, we see that all the typical GIT quotients $(\pp^1)^n\sslash_\ob PGL_2$ are of Fano type over the complex numbers.
\end{rmk}

\subsection{The Mori chamber decomposition of {{$Bl_{n-1}\pp^{n-3}$}}}
The key result of this subsection is presented in Proposition \ref{bijection1}. In \cite[Sec. 2.47]{Bolognesi_Massarenti}, there is a bijection between the VGIT decomposition of $(\pp^1)^n$ with respect to the action of $PGL_2$ and the Mori chamber decomposition of $Bl_{n-1}\pp^{n-3}$, which is explicitly computed in \cite[Theorem 2.32]{Bolognesi_Massarenti}. In Lemma \ref{compositionlemma} and Proposition \ref{bijection1}, we reprove this bijection from a different perspective, identifying the birational contractions of $Bl_{n-1}\pp^{n-3}$ with the GIT quotients arising from the corresponding VGIT chambers.

\begin{notation}\label{notazione}
    For the rest of this section, we set $X=Bl_{n-1}\pp^{n-3}$, $Y=(\pp^1)^n$. We choose a typical linearization $\ob=(b_1,\dots,b_n)$ for the action of $PGL_2$ on $Y$ such that $\Sigma_\ob=(\pp^1)^n\sslash_\ob PGL_2\cong X$. Denote by $Y^{ss}=Y^{ss}(\ob)$ the semi-stable locus of the action, and let $p:Y\dashrightarrow\Sigma_\ob$ be the projection to the quotient. Set also $\oc=(\frac{1}{n-3},\dots,\frac{1}{n-3},1)$, so that $\overline{M}_\oc\cong Bl_{n-1}\pp^{n-3}$. We will use the following natural bases:
\begin{enumerate}
    \item\label{basis1} for $\pic(X)_\qq$, we have $\{E_1,\dots,E_{n-1},H\}$, where the $E_i$ are the exceptional divisors and $H$ is the pullback of a hyperplane class of $\pp^{n-3}$, and
    \item for $\pic(Y)_\qq$, we have $\{L_1,\dots,L_n\}$, where $L_i=\pi_i^*H$ and $\pi_i:Y\to\pp^1$ is the projection on the $i$-th factor. 
\end{enumerate}
We know that such a linearization $\ob$ exists by Proposition \ref{proptypical} and lies in the closure of the coarse chamber of $\oc$. It can be explicitly constructed in the following way: take a reduction morphism $\overline{M}_\oc\to\overline{M}_\oa$ with $\overline{M}_\oa$ small as in the proof of Lemma \ref{lemmasmall}.(\ref{lm2}), and then set $\ob=N(\oa)$ as in the proof of Proposition \ref{proptypical}.(\ref{proptyp4}). 
After an explicit computation, it is easy to see that $\codim(Y\setminus Y^{ss},Y)\ge2$. For this reason, $\pic(Y^{ss})_\qq\cong\pic(Y)_\qq$ and the basis $\{L_1,\dots,L_n\}$ can also be used for $\pic(Y^{ss})_\qq$.
\end{notation}

\begin{lemma}\label{compositionlemma}
    Assuming the setup of Notation \ref{notazione}, there is a composition of isomorphisms
    \begin{equation}\label{composition}
\phi:\pic(X)_{\qq}\xrightarrow{\sim}\pic(\overline{M}_\oc)_{\qq}\xrightarrow{\rho^{*-1}}\pic(\Sigma_\ob)_\qq\xrightarrow{p^{*}}\pic(Y)_\qq,\end{equation} 
where the first arrow is the Kapranov model as in Example \ref{exom}, $\rho:\overline M_\oc\xrightarrow{\sim} \Sigma_\ob$ is given by Proposition \ref{proptypical} and $p:Y^{ss}\to\Sigma_\ob$ is the projection to the quotient. 
Moreover, $\phi$ maps
\[
H\mapsto (n-3)L_n+\sum_{i\ne n} L_i,\qquad E_i\mapsto L_i+L_n\qquad{and}\qquad\widetilde{\Lambda}_{ij}\mapsto L_i+L_j.
\]
\end{lemma}

\begin{proof}
We can treat $\rho$ as if its target is a small Hassett space $\overline{M}_\oa$ isomorphic to $\Sigma_\ob$. All spaces in the statement have a universal family with $n$ sections. On each of these spaces, the class $\psi_i$ is defined as the first Chern class of the pullback of the relative dualizing sheaf of the universal family via the $i$-th section. We denote the boundary divisors by $\delta_I$ as in Definition \ref{boundariesdef}. Using \cite[Lemma 2.3]{CT2}, we find that $\rho^*\psi_n=\psi_n-\sum_{i\ne n}\delta_{\{i,n\}}$ and $\rho^*\delta_{\{i,n\}}=\delta_{\{i,n\}}$. Therefore, 
\[
\rho^{*-1}(\delta_{\{i,n\}})=\delta_{\{i,n\}}=-\frac{1}{2}(\psi_i+\psi_n)\:\text{ for $i\ne n$, and }\:\]
\[
\rho^{*-1}(\psi_n)=\psi_n+\sum_{i\ne n}\delta_{\{i,n\}}=\psi_n-\frac{1}{2}\sum_{i\ne n}(\psi_i+\psi_n)
.\]

We now understand how $L_i$ descends to $\pic(\Sigma_\ob)_\qq$. As seen in the proof of Proposition \ref{proptypical}.(\ref{proptyp1}), the GIT quotient $\Sigma_\ob$ carries a universal family $\pi:\ou\to\Sigma_{\ob}$ that parametrizes $\ob$-stable curves and there is a pullback diagram
\[ 
\begin{tikzcd}    Y^{ss}\times\pp^1\arrow[r,"\widetilde p"]\dar&\dar\ou\\
    Y^{ss}\arrow[r,"p"]&\Sigma_\ob,
\end{tikzcd},
\] 
where $Y^{ss}\times\pp^1\to Y^{ss}$ is the trivial family with $n$ canonical sections $s_1,\dots,s_n$ defined as $s_i(p_1,\dots,p_n)=((p_1,\dots,p_n),p_i)$, and $\widetilde p$ is given by the universality of $\ou$. In $Y^{ss}$, we have
\[
\psi_i=s_i^*K_{Y^{ss}\times\pp^1 / Y^{ss}}=-2\pi_i^*H_{\pp^1}=-2L_i.\]
As the diagram is a pullback, $p^*\psi_i=\psi_i=-2L_i$. Since in our setting the descent is unique, we conclude that the last arrow in (\ref{composition}) maps $\psi_i$ to $-2L_i$. 
Finally, the composition $p^*\circ\rho^{*-1}$ is such that
\[
p^*\circ\rho^{*-1}(\psi_n)= (n-3)L_n+\sum_{i\ne n} L_i\qquad\textup{and}\qquad p^*\circ\rho^{*-1}(\delta_{\{i,n\}})= L_i+L_n.
\]
Recall from Example \ref{exom} that the isomorphism $X\cong\overline{M}_\oc$ makes the boundary divisors $\delta_{\{i,n\}}$ correspond with the exceptional divisors $E_i$. Moreover, by pulling back $\psi_n$ via the Kapranov reduction morphism $\overline{M}_\oc\to\overline{M}_{\oa_0}\cong\pp^{n-3}$ (cf. Example \ref{exom}), with \cite[Lemma 2.1]{CT2}, we see that $H=\psi_n$ in $X$, that $\phi(H)=(n-3)L_n+\sum_{i\ne n} L_i$ and that $\phi(E_i)=L_i+L_n$. That $\phi(\widetilde{\Lambda}_{ij})=L_i+L_j$ follows from $\widetilde{\Lambda}_{ij}\equiv H-E_i-E_j$.
\end{proof}

Recall from Example \ref{extremeblowup} that the extremal rays of $\eff(X)$ are precisely the exceptional divisors $E_i$ and the strict transforms of the hyperplanes $\widetilde{\Lambda}_{ij}$ spanned from all the points but two. From Lemma \ref{compositionlemma}, it is easy to see that $\phi(\mathrm{int}(\eff(X)))\subset\mathrm{Amp}(Y)$. Recall the map $N:\mathrm{Amp}(Y)\to\mathcal{V}_n$ from Notation \ref{vgitdomain}. Thus, the composition $N\circ\phi$ is well-defined.
\begin{definition}\label{mapping}
We define the map
    \begin{equation*}
        \Phi=N\circ\phi:\mathrm{int}(\eff(X))\to\mathcal{V}_n.
    \end{equation*}
\end{definition} 
\begin{proposition}\label{bijection1}
Let $\Phi$ be the map of Definition \ref{mapping}.
\begin{enumerate}
    \item The map $\Phi$ induces a bijection between the Mori chambers of $X$ and the chambers of the VGIT decomposition of $(\pp^1)^n$ with the action of $PGL_2$. 
    \item If $C$ is a Mori chamber of $\eff(X)$, the target of the birational contraction $\psi_C:X\dashrightarrow Y$ associated with $C$ is isomorphic to $(\pp^1)^n\sslash _{\Phi(C)}PGL_2$.
    \item The Mori chambers of $\eff(X)$ are the preimages of the VGIT chambers of $\mathcal{V}_n$ via $\Phi$. If $D=hH+\sum_{i=1}^{n-1} e_iE_i$, the walls are \[(|I|-1)h+\sum_{i\in I}e_i=0\text{ for }I\subset\{1,\dots,n-1\}
    , I\ne\varnothing.
    \]
\end{enumerate}
\end{proposition}
\begin{proof}
    If $D$ is a divisor in $\eff(X)$, by Lemma \ref{abstractlemma} presented below, there is a birational contraction $\psi_D:X\cong\Sigma_\ob\dashrightarrow\Sigma_{\Phi(D)}$. We notice that $\psi_D$ is constant on the preimage of the VGIT chamber of $\Phi(D)$. As a consequence, the preimages of the VGIT chambers of $\mathcal{V}_n$ are the Mori chambers of $X$.
\end{proof}

\begin{lemma}\label{abstractlemma}
     Let $Y$ be a normal projective variety with the action of a reductive group $G$. Let $\ol\in\pic^G(Y)$ be an ample line bundle such that $\mathrm{codim}(Y\setminus Y^{ss}(\ol),Y)\ge2$ and the action of $G$ on $Y^{ss}(\ol)$ is free. Then, for any ample $\ol_2\in\pic^G(Y)$, there is a birational contraction $Y\sslash_\ol G\dashrightarrow Y\sslash_{\ol_2}G$.
\end{lemma}
\begin{proof}
Let $\pi:Y\dashrightarrow Y\sslash_\ol G$ be the projection to the quotient. As a consequence of Kempf's Descent Lemma \cite[Théorème 2.3]{refperkempf}, there is an isomorphism $\pic(Y\sslash_\ol G)\xrightarrow{\pi^*}\pic^G(Y)$, where we have identified $\pic^G(Y^{ss}(\ol))\cong\pic^G(Y)$. Up to choosing a multiple of $\ol_2$, we can find an effective divisor $D$ of $Y\sslash_\ol G$  such that $\ol_2=\pi^*\oo(D)$. Since $R(Y,\ol_2)\cong R(Y^{ss}(\ol),\pi^*D)$, we have
\begin{align*}
    Y\sslash_{\ol_2}G&=\proj R(Y,\pi^*D)^G\cong\proj R(Y^{ss}(\ol),\pi^*D)^G\cong\\&\cong\proj R(Y\sslash _{\ol}G,\pi_*(\pi^*D))^G=\proj R(Y\sslash _{\ol}G,D).
\end{align*}
Therefore, the linear system $|D|$ gives a rational map $\psi_D:Y\sslash _{\ol}G\dashrightarrow Y\sslash _{\ol_2}G$. Up to choosing multiples of $\ol$, we can assume that $R(Y\sslash_\ol G,D)$ is generated in degree 1 by sections $s_1,\dots,s_N$. As $p^*:R(Y\sslash_\ol G,D)\xrightarrow{\sim}R(Y,\ol_2)^G$ is an isomorphism, the ring $R(Y,\ol_2)^G$ is generated in degree 1 by sections $p^*s_1,\dots,p^*s_N$. Moreover, if we let $U=Y^{ss}(\ol)\cap Y^{ss}(\ol_2)$ and  $\pi_2:Y\dashrightarrow Y\sslash_{\ol_2} G$, for every $x\in U$ we have \[\psi_D\circ\pi(x)=[s_1(\pi(x)):\dots:s_N(\pi(x))]=[\pi^*s_1(x):\dots:\pi^*s_N(x)]=\pi_2(x).
\]
Hence, we prove that $\psi_D\circ \pi|_U=\pi_2|_U$, the restriction of $\psi_D$ to $\pi(U)$ is an isomorphism and $\psi_D$ is birational. That it is a birational contraction follows from \cite[Lemma 1.6]{HK}.
\end{proof}

\begin{proof}[Proof of Proposition {{\ref{proposition}}}]
Parts (\ref{pr1}) and (\ref{pr2}) follow from combining Proposition \ref{bijection1} and Proposition \ref{proptypical}.(\ref{proptyp4}). (\ref{pr3}) follows since the birational contractions of a Mori dream space are Mori dream spaces (see Proposition \ref{MDSprop}). As small Hassett spaces are smooth, we have (\ref{pr5}).

We prove (\ref{pr6}). Let $\oc$ be a tuple such that $\overline M_\oc\cong Bl_{n-1}\pp^{n-3}$, let $\overline M_\oa$ be small and let $\psi:\overline M_\oc\dashrightarrow\overline M_\oa$ be the associated birational contraction. As $\psi$ is an isomorphism on the configuration space $ M_{0,n}$, there is a commutative diagram \begin{equation}\label{situation2}
    \begin{tikzcd}
    &\overline M_{0,n}\arrow[dl,"q",swap]\arrow[dr,"p"]&\\
    \overline M_{\oc}\arrow[rr,"\psi",dashed]&&\overline M_{\oa}
    \end{tikzcd},\end{equation}
where $q$ and $p$ are reduction morphisms. Recall from Remark \ref{divisors} that, if $|I|=2$ or $n-2$, the divisor $\delta_I$ is mapped by $q$ to a type-2 divisor and not contracted. On the other hand, if $2<|I|<n-2$, the dimension of $\delta_I$ drops in $\overline M_{\oc}$, and $\delta_I$ is contracted by $q$. Thus, if we ask that all divisors $\delta_I$ with $2<|I|<n-2$ do not get contracted by $p$, we have that $\psi$ is an SQM. 

It has been proved in \cite{massarentiaraujo} that, over the complex numbers, the variety $Bl_{n-1}\pp^{n-3}$ is of Fano type for all $n$. Over the complex numbers, it follows from its properties and \cite[Corollary 5.5]{surjimageslogfano} that being of Fano type for a Mori dream space is preserved under its birational contractions. As small Hassett spaces are the birational contractions of $X$, we deduce (\ref{pr4}).
\end{proof}

\begin{rmk}
Given a small Hassett space $\overline{M}_\oa$, a commutative diagram as (\ref{situation2}) can actually be found without using Proposition \ref{bijection1} and GIT quotients. Indeed, the two reduction morphisms $p$ and $q$ induce a birational map $\overline M_{\oc}\dashrightarrow \overline M_{\oa}$. By determining which boundary divisors get contracted by $q$ and $p$ as in the proof of \ref{proposition}.(\ref{pr6}), one proves that the map $\overline M_{\oc}\dashrightarrow \overline M_{\oa}$ is a birational contraction. However, this procedure does not give a complete picture of the birational contractions of $Bl_{n-1}\pp^{n-3}$ as Proposition \ref{proposition} does. 
\end{rmk}
\begin{rmk}
    It is worth noting that one could alternatively use the classical Gelfand–MacPherson correspondence \cite[Theorem 2.2.4]{kapranov}, which provides a natural identification between quotients $(\pp^1)^n\sslash PGL_2$ and torus quotients of the Grassmannian $\mathrm{Gr}(2,n)$. For these torus quotients, \cite[Corollary 2.4]{HK} can be adapted and used to find the birational contractions of $X$. As a matter of fact, Lemma \ref{abstractlemma} uses some specific hypotheses similar to \cite[Theorem 2.3]{HK} to obtain the desired conclusion in our setting.    
\end{rmk}

\section{The Mori cone of small Hassett spaces}
\label{Lastsec}

In this section, we prove that the effective and the Mori cone of small Hassett spaces are generated by strata. The generation of the effective cone is induced by $Bl_{n-1}\pp^{n-3}$ and its birational contractions. Moreover, this highlights another difference between small Hassett spaces and  $\overline{M}_{0,n}$, whose effective cone is not generated by 1-codimensional strata for $n\ge6$ (see \cite{Hassett2001OnTE}), and not even finitely generated for $n\ge8$ (see \cite{mullane2025isoresidualfibrationsbirationalgeometry}).

\begin{theorem}\label{mlkd}
    Let $\overline M_{\oa}$ be a small Hassett space. Then, $\overline{\textup{Eff}}(\overline M_{\oa})$ is finitely generated by 1-codimensional strata. Moreover, 1-codimensional strata are precisely the extremal rays if and only if there are no subsets $I\subset\{1,\dots,n\}$ with $|I|=n-2$ and $\sum_{i\in I}a_i\le1$.
\end{theorem}

\begin{proof}  
Set $\oc=(\frac{1}{n-3},\dots,\frac{1}{n-3},1)$. Consider $f:Bl_{n-1}\pp^{n-3}\cong\overline M_{\oc}\dashrightarrow \overline M_{\oa}$ as given by Proposition \ref{proposition}. We stated in Example \ref{extremeblowup} that $\overline{\textup{Eff}}(\overline M_{\oc})$ has 1-codimensional strata as extremal rays. Since $\overline M_{\oc}$ is a Mori dream space and $f$ is a birational contraction, the images of the extremal rays of $\overline{\textup{Eff}}(\overline M_{\oc})$ through $f$ will generate $\overline{\textup{Eff}}(\overline M_{\oa})$. Since reduction morphisms map strata to strata, if not contracted, these images are also 1-codimensional strata, and the first part is proved.

With the additional hypothesis on $\oa$, by Proposition \ref{proposition}.(\ref{pr6}) we have that $f$ is an SQM. Hence, $\overline{\textup{Eff}}(\overline M_{\oa})$ is isomorphic to $\overline{\textup{Eff}}(\overline M_{\oc})$ and 1-codimensional strata are its extremal rays.
\end{proof}

We report a lemma from \cite{keel1996contractibleextremalraysoverlinem0n} that is useful for obtaining information on the Mori cone when the generators of the effective cone are known.
\begin{lemma}[{{\cite[Corollary 2.3]{keel1996contractibleextremalraysoverlinem0n}}}]\label{keelmc}
    Let $X$ be a projective variety and let $D\subset X$ be a reduced effective divisor whose irreducible components span $\overline{\textup{Eff}}(X)$, and let $W(D)\subset\overline{NE}(X)$ be the closure of the cone generated by the classes of curves in $D$. Then, $W(D)=\overline{NE}(X)$. 
\end{lemma}
\begin{cor}\label{lemma2}
    Let $\overline M_{\oa}$ be a small Hassett space. Then, $\overline{NE}(\overline M_{\oa})$ is finitely generated by 1-dimensional strata.
\end{cor}
\begin{proof}
    We proceed by induction on the number of marked points $n$. For $n=3,4$ the assertion is trivial. By Theorem \ref{mlkd} we know that $\overline{\textup{Eff}}(\overline M_{\oa})$ is finitely generated by $1$-codimensional strata. Let $\Delta$ be the union of all the boundary divisors. If we set $W(\Delta)$ to be the closure in $\overline{NE}(\overline{M}_{\oa})$ of the cone spanned by the curves lying on the divisors $\Delta$, by Lemma \ref{keelmc} we have    \begin{equation}\label{equa}
    \overline{NE}(\overline M_{\oa})=W(\Delta).
    \end{equation} 
    Since $\overline M_{\oa}$ is small, by Lemma \ref{lemmasmall}.(\ref{lm3}), 1-codimensional strata in $\overline M_{\oa}$ are isomorphic to small Hassett spaces with $n-1$ marked points. Let $I=\{i_1,i_2\}$ and let $\delta_I\cong\overline{M}_{\oa'}$, with the weight tuple \[\oa'=(a_{i_1}+a_{i_2},a_1,\dots,\hat{a}_{i_1},\dots,\hat{a}_{i_2},\dots,a_n).\] By inductive hypothesis, $\overline{NE}(\overline M_{\oa'})$ is generated by its 1-dimensional strata. By putting all the pieces in equation (\ref{equa}), we find that $W(\Delta)$ is the closed cone generated by $1$-dimensional strata in $\overline M_{\oa}$.
\end{proof}
\begin{rmk}
    The proof of Corollary \ref{lemma2} relies on two facts:
    \begin{enumerate}
        \item the effective cone of small Hassett spaces is generated by strata, and
        \item small Hassett spaces form an inductive family, in the sense that their strata are themselves isomorphic to small Hassett spaces.
    \end{enumerate}
    These two facts, together with Lemma \ref{keelmc}, make the induction argument possible. 
    The next natural case of Hassett spaces to consider is when they parametrize stable curves with at most two irreducible components. These Hassett spaces form an inductive family. However, their effective cone is not generated by 1-codimensional strata in general, as $\overline M_{0,6}$ is of this kind. An inductive argument might still be possible if the effective cone turns out to be finitely generated with generators of a known form. In conclusion, we leave open the question for this family. 
\end{rmk}

\printbibliography
\let\thefootnote\relax\footnotetext{\textit{Institutional email address}: \url{davide.ricci@uvsq.fr}

\textit{Permanent email address}: \url{davide.ricci314@gmail.com}}
\end{document}